\documentclass{amsart}
\usepackage{amssymb,amstext,amsmath,amscd,amsthm,amsfonts,enumerate,graphicx,latexsym,stmaryrd,multicol}

\usepackage[usenames]{color}
\usepackage[all]{xy}
\newtheorem{thm}{Theorem}[section]
\newtheorem{lem}[thm]{Lemma}
\newtheorem{prop}[thm]{Proposition}
\newtheorem{cor}[thm]{Corollary}
\theoremstyle{definition}
\newtheorem{dfn}[thm]{Definition}

\newtheorem{conj}[thm]{Conjecture}

\newtheorem{rmk}[thm]{Remark}
\theoremstyle{remark}

\newtheorem*{ac}{Acknowlegments}
\newtheorem*{conv}{Convention}

\numberwithin{equation}{thm}

\def\depth{\operatorname{depth}}
\def\Ext{\operatorname{Ext}}

\def\ge{\geqslant}

\def\Hom{\operatorname{Hom}}
\def\image{\operatorname{Im}}

\def\m{\mathfrak{m}}
\def\mod{\operatorname{mod}}
\def\Mod{\operatorname{Mod}}

\def\n{\mathfrak{n}}
\def\nf{\mathrm{NF}}

\def\p{\mathfrak{p}}
\def\pd{\operatorname{pd}}

\def\spec{\operatorname{Spec}}
\def\supp{\operatorname{Supp}}

\def\Tor{\operatorname{Tor}}
\def\tr{\operatorname{Tr}}

\def\m{\mathfrak{m}}
\def\C{\mathcal{C}}
\def\D{\mathcal{D}}
\def\X{\mathcal{X}}
\def\Y{\mathcal{Y}}

\def\B{\mathcal{B}}
\def\A{\mathcal{A}}
\def\res{\operatorname{res}}
\def\inf{\operatorname{inf}}

\def\radius{\operatorname{radius}}
\def\id{\operatorname{id}}
\def\add{\operatorname{add}}
\def\l{\operatorname{length}}
\def\id{\operatorname{id}}
\def\dim{\operatorname{dim}}
\def\Ann{\operatorname{Ann}}
\def\tG{\operatorname{G}}
\def\seek{\operatorname{seek}}
\begin{document}
\allowdisplaybreaks
\title{On the finiteness of radii of resolving subcategories}
\author{Yuki Mifune}
\address{Graduate School of Mathematics, Nagoya University, Furocho, Chikusaku, Nagoya 464-8602, Japan}
\email{yuki.mifune.c9@math.nagoya-u.ac.jp}
\begin{abstract}
Let $R$ be a commutative noetherian ring. Denote by $\mod R$ the category of finitely generated $R$-modules. In this paper, we investigate the finiteness of the radii of resolving subcategories of $\mod R$ with respect to a fixed semidualizing module. As an application, we give a partial positive answer to a conjecture of Dao and Takahashi: we prove that for a Cohen--Macaulay local ring $R$, a resolving subcategory of $\mod R$ has infinite radius whenever it contains a canonical module and a non-MCM module of finite injective dimension.
\end{abstract}
\maketitle
\section{Introduction}
Let $R$ be a commutative noetherian ring. Denote by $\mod R$ the category of finitely generated $R$-modules. The concept of the radius of a full subcategory of $\mod R$ has been introduced by Dao and Takahashi \cite{Dao Takahashi}. It has been linked to many well-studied notions such as the dimension of the stable category of maximal Cohen--Macaulay modules, finite/countable Cohen--Macaulay representation type and the uniform Auslander condition. Among other things, Dao and Takahashi proved the following result, which forms an essential part of the main result of their paper.

\begin{thm}[Dao and Takahashi] \label{4.9}
Let $R$ be a local ring. Let $\X$ be a resolving subcategory of $\mod R$. If $\X$ contains a module $M$ such that $0<\pd_R M <\infty$, then $\X$ has infinite radius. 
\end{thm}
Here, the notion of a resolving subcategory has been introduced by Auslander and Bridger \cite{AB} in the 1960s, and has been studied widely and deeply so far; see \cite{APST, AR, KS, stcm, arg} for instance. In view of this theorem, Dao and Takahashi presented the following conjecture.

\begin{conj}[Dao and Takahashi] \label{3.1}
Let $R$ be a Cohen--Macaulay local ring. Let $\X$ be a resolving subcategory of $\mod R$. If $\X$ contains a non-MCM module $M$, then $\X$ has infinite radius.
\end{conj}
The main purpose of the present paper is to prove the following theorem, which says that Conjecture \ref{3.1} does hold if $\X$ contains a canonical module and the module $M$ has finite injective dimension.

\begin{thm}[Corollary \ref{cor3.2}] \label{thm1.1} 
Let $R$ be a Cohen--Macaulay local ring with a canonical module $\omega$. Let $\X$ be a resolving subcategory of $\mod R$. If $\X$ contains $\omega$ and a non-MCM module $M$ with $\id_R M <\infty$, then $\X$ has infinite radius.
\end{thm}
In fact, this theorem is a corollary of the next theorem, which  simultaneously recovers Theorem \ref{4.9} stated above; for the notation and terminology, see Definitions \ref{defofcpd} and \ref{defofcres}(1).

\begin{thm} [Theorem \ref{thm3.1}] \label{thm1.2}
Let $C$ be a semidualizing $R$-module and $\X$ a $C$-resolving subcategory of $\mod R$. If $\X$ contains a module $M$ such that $0<\pd_C M <\infty$, then $\X$ has infinite radius.
\end{thm}
The organization of this paper is as follows. In section 2 we state basic definitions and their properties for later use. In section 3 we state and prove various propositions, and by using them we give proofs of the above theorems. In section 4 we give an application of Theorem \ref{thm1.1}, which deals with the case where the Cohen--Macaulay ring $R$ does not have a canonical module.
\begin{conv}

Throughout this paper, all rings are commutative noetherian local rings, all modules are finitely generated, and all subcategories are full and strict, i.e., closed under isomorphism. Let $R$ be a (commutative noetherian local) ring with maximal ideal $\m$ and residue field $k$. 

\end{conv}

\section{Preliminaries}

In this section, we recall some definitions and basic properties of them needed later. We begin with the notion of semidualizing modules. 

\begin{dfn}
An $R$-module $C$ is said to be {\em semidualizing} if the following hold. \\
(1) The natural homomorphism $R \to \Hom_R(C,C)$ is an isomorphism. \\
(2) One has $\Ext^{i}_R(C,C)= 0 $ for all $i>0$.
\end{dfn}

Note that $R$ is itself a semidualizing $R$-module. If $R$ is Cohen--Macaulay local ring with canonical module $\omega$, then $\omega$ is a semidualizing $R$-module. \par
The proposition below provides the basic properties of a semidualizing module.

\begin{prop} \label{prop2.1}
Let $C$ be a semidualizing $R$-module. Then the following hold. 
\begin{enumerate}[\rm(1)]
\item
One has $\supp C = \spec R$. 
\item
Let $R \to S $ be a flat homomorphism of rings. Then $C\otimes_R S $ is a semidualizing $S$-module. 
\item
The equality $\depth C = \depth R$ holds. 
\end{enumerate}
\end{prop}

\begin{proof}
The statements (1) and (2) are easy to see by definition. The statement (3) is shown in \cite{Golod}.
\end{proof}

Let $\X$ be a subcategory of $\mod R$. We denote by $\add \X$ (or $\add_R \X$) the {\em additive closure of $\X$}, namely, the subcategory of $\mod R$ consisting of direct summands of finite direct sums of modules in $\X$. When $\X$ consists of a single module $X$, we simply denote it by $\add X$ (or $\add_R X$). 
\begin{rmk}
If $C$ is a semidualizing $R$-module, then $\add_R C$ consists all modules of the form $C^{\oplus n}$ with $n\geq0$. Indeed, suppose that $N$ is a direct summand of $C^{\oplus m} $. Then $\Hom_R(C,N)$ is a direct summand of $R^{\oplus m}$. This implies that $\Hom_R(C,N)$ is free $R$-module because $R$ is local. Hence we have $\Hom_R(C,N) \cong R^{\oplus n} $ for some $n\geq 0$ and applying the functor $C\otimes_R -$, we get the isomorphisms $N \cong C\otimes_R \Hom_R(C,N) \cong C^{\oplus n}$.
\end{rmk}

Next we recall the definition of projective dimension relative to a fixed semidualizing module.

\begin{dfn} \label{defofcpd}
Let $C$ be a semidualizing $R$-module and $M$ an $R$-module. Suppose that $M$ is nonzero. The {\em $C$-projective dimension of $M$}, denoted $\pd_C M$, is defined as the infimum of the lengths of ($\add_R C$)-resolutions, namely, exact sequences of the form $0\to C_n \to C_{n-1} \to \cdots \to C_1 \to C_0 \to M \to 0$ with each $C_i$ belonging to $\add_R C$.
For convention, we set $\pd_C 0 = -\infty$.
\end{dfn}

The following proposition shows that $C$-projective dimension can be replaced with ordinary projective dimension by applying the functor $\Hom_R (C,-)$

\begin{prop} \label{prop2.3}
Let $C$ be a semidualizing $R$-module and $M$ an $R$-module. Then we have $\pd_C M = \pd_R \Hom_R(C, M)$.
\end{prop}
\begin{proof} The assertion follows from \cite[Theorem 2.11]{Takahashi White}.
\end{proof}

We recall the notions of the Bass class and the Auslander class with respect to a semidualizing module and basic properties of them.

\begin{dfn}
Let $C$ be a semidualizing $R$-module. The {\em Bass class with respect to C}, denoted $\B_C$, is a subcategory of $\mod R$ consisting of modules $M$ satisfying the following two conditions.
\begin{enumerate}
\item
One has $\Ext^{i}_R(C,M)= \Tor^R_{i}(C, \Hom_R(C,M))= 0$ for all $i>0$. 
\item
The natural homomorphism $M\to\Hom_R(C,C\otimes_R M)$ is an isomorphism.
\end{enumerate}
Dually, the {\em Auslander class with respect to C}, denoted $\A_C$, is a subcategory of $\mod R$ consisting of modules $M$  satisfying the following two conditions.
\begin{enumerate}
\item
One has $\Tor^R_{i}(C, M)= \Ext^{i}_R(C, C\otimes_R M)=0 $ for all $i>0$.
\item 
The natural homomorphism $M\to\Hom_R(C, C\otimes_R M)$ is an isomorphism.
\end{enumerate} 
\end{dfn}

\begin{rmk}
The notions of the Bass class and the Auslander class are usually defined as subcategories of the category $\Mod R$ of all $R$-modules. Thus the Bass and Auslander classes in our sense are the restrictions of the usual ones to $\mod R$. 
\end{rmk}

The equivalence between the Bass and Auslander classes given below is often called the {\em Foxby equivalence}.

\begin{prop} \label{prop2.2}

Let $C$ be a semidualizing $R$-module. Then the following hold. 
\begin{enumerate} [\rm(1)]
\item
The functors $\Hom_R(C,-):\B_C \to \A_C$ and $C\otimes_R - : \A_C \to \B_C$ are category equivalences. 
\item
The subcategories $\B_C$ and $\A_C$ are closed under direct summands. 
\item
The subcategories $\B_C$ and $\A_C$ have the property that if two of the three modules in a short exact sequence in $\mod R$ are in the subcategory, then so is the third.
\end{enumerate} 
\end{prop}

\begin{proof}
The statements (1) and (2) are routine to check. The statement (3) is shown in \cite[Lemma 1.3]{Foxby}.
\end{proof}

Note that if an $R$-module $M$ has finite $C$-projective dimension, $M$ belongs to $\B_C$, since $\B_C$ is closed under cokernels of monomorphisms and contains $\add C$. \par
We recall the definitions of a resolving subcategory and the resolving closure of a subcategory of $\mod R$.
\begin{dfn} \label{defofcres}
Let $C$ be a semidualizing $R$-module.
\begin{enumerate}[\rm(1)]
\item
A subcategory $\X$ of $\mod R$ is {\em $C$-resolving} if $\X$ contains $C$ and is closed under direct summands, extensions, and kernels of epimorphisms in $\mod R$. 
\item
For a subcategory $\X$ of $\mod R$, we denote by $\res_{C} \X$ the {\em $C$-resolving closure of $\X$}, namely the smallest $C$-resolving subcategory of $\mod R$ containing $\X$. Similarly, we denote by $\seek \X$ the smallest subcategory of $\mod R$ which contains $\X$ and is closed under direct summands, extensions and kernels of epimorphisms in $\mod R$.
\end{enumerate}
\end{dfn}

Now we recall the definition of the radius of a subcategory of $\mod R$.
\begin{dfn}Let $M$ be an $R$-module. 
\begin{enumerate}[\rm(1)]
\item
Take a minimal free resolution $\cdots \xrightarrow{\delta_{n+1}} F_n \xrightarrow{\delta_n} F_{n-1} \xrightarrow{\delta_{n-1}}\cdots \xrightarrow{\delta_1}F_0 \to M \to 0$ of $M$. Then for each $n\geq 1$, the image of $\delta_n$ is called the {\em $n$-th syzygy of $M$} and denoted by $\Omega^n M$. For convention, we set $\Omega^0 M = M$.
\item
For a subcategory $\X$ of $\mod R$ we denote by $\lbrack \X \rbrack$ the additive closure of the subcategory consisting of $R$ and all modules of the form $\Omega^i X$, where $i\geq0$ and $X\in \X$. When $\X$ consists of a single module $X$, we simply denote it by $\lbrack X \rbrack$. 
\item
For subcategories $\X, \Y$ of $\mod R$ we denote by $\X \circ \Y$ the subcategory of $\mod R$ consisting of $R$-modules $Z$ which fits into an exact sequence $0 \to X \to Z \to Y \to 0$ with $X\in \X$ and $Y\in \Y$. We set $\X \bullet \Y = \lbrack\lbrack \X \rbrack \circ \lbrack \Y \rbrack \rbrack$. 
\item
Let $\X$ be a subcategory of $\mod R$. We define the {\em ball of radius $r$ centered at $\X$} as 
\begin{equation}
\lbrack\X\rbrack_r = \nonumber
\begin{cases}
\lbrack\X\rbrack & \text{if $r=1$,} \\
\lbrack\X\rbrack_{r-1} \bullet \lbrack\X\rbrack & \text{if $r>1$.}
\end{cases}
\end{equation}
If $\X$ consists of a single module $G$, then we simply denote $\lbrack\X\rbrack_r$ by $\lbrack G\rbrack_r$, and call it the ball of $\radius r$ centered at $G$. We write $\lbrack\X\rbrack^R _r$ when we should specify that $\mod R$ is the ground category where the ball is defined. 
Note that the operator “ $\bullet$ " satisfies the associative law by \cite[Proposition 2.2]{Dao Takahashi}. 
\item
Let $\X$ be a subcategory of $\mod R$. We define the {\em radius} of $\X$, denoted by $\radius\X$, as the infimum of the integers $n\geq 0$ such that there exists a ball of $\radius n+1$ centered at a module containing $\X$, that is, 
\begin{equation} \label{radiusdef}
\radius\X=\inf\{n\ge0\mid\X\subseteq[G]_{n+1}\text{ for some }G\in\mod R\}.
\end{equation}
\end{enumerate}
\end{dfn}

\begin{rmk} 
Let $C$ be a semidualizing $R$-module. 
Suppose $M$ belongs to $\B_C$ and take a minimal free resolution $F$ of $\Hom_R(C,M)$. Applying the functor $C\otimes_R  -$, we get the {\em minimal $C$-resolution} $C\otimes_R F$ of $M$ and can define the {\em $i$-th $C$-syzygy}  $\Omega^{i}_{C} M$ of $M$ as the image of each differential map. Similarly, for a subcategory $\X$ of $\mod R$, we can consider the additive closure $[\X]^C$ of the subcategory consisting of $C$ and all modules of the form $\Omega^{i}_{C} X $, where $i\geq0$ and $X\in \X$. Hence we can define the {\em $C$-radius} of $\X$, denoted by $\radius_C \X$, as the infimum of the integers $n\geq 0$ such that there exists a {\em $C$-ball} of $\radius n+1$ centered at a module which belongs to $\B_C$ such that $\X$ is contained in the $C$-ball. However, these concepts are only a formality and cannot lead to consequence we want; see \ref{rmk3.1}.  
\end{rmk}

We close this  section by defining  the transpose of a module with respect to a fixed semidualizing module.
\begin{dfn} Let $M$ be an $R$-module. Take a minimal free resolution $\cdots \xrightarrow{\delta_{n+1}} F_n \xrightarrow{\delta_n} F_{n-1} \xrightarrow{\delta_{n-1}}\cdots \xrightarrow{\delta_2} F_1 \xrightarrow{\delta_1}F_0 \to M \to 0$ of $M$. 
\begin{enumerate}[\rm(1)]
\item
The cokernel of the $R$-dual map $\delta^* _1 : F^*_0 \to F^*_1$ , where $(-)^* = \Hom_R(-,R)$,  is called the ({\em Auslander}) {\em  transpose} of $M$ and denoted by $\tr_R M$. 
\item
Let $C$ be a semidualizing $R$-module. Similarly, the cokernel of the $C$-dual map $\delta^{\dagger} _1 : F^{\dagger}_0 \to F^{\dagger}_1$ , where $(-)^{\dagger} = \Hom_R(-,C)$, is called the {\em transpose} of $M$ {\em with respect to $C$} and denoted by $\tr_C M$. 
Note that for any finitely generated projective $R$-module $P$, there exists a functorial isomorphism $C\otimes_R P^* \cong P^{\dagger}$. Hence we have $\tr_C M \cong C\otimes_R \tr_R M$.
\end{enumerate}
\end{dfn}

\section{Proof of main theorem}
In this section we give the proofs of Theorem \ref{thm1.1} and \ref{thm1.2}.
We introduce some notation for convenience. Let $\X$ be a subcategory of $\mod R$. For a prime ideal $\p$ of $R$, we denote by $\X_{\p}$ the subcategory of $\mod R_{\p}$ consisting of all modules of the form $X_{\p}$, where $X\in \X$. In the same way, we denote by $\widehat{\X}$ the subcategory of $\mod \widehat{R}$ consisting of all modules of the form $\widehat{X}$, where $X \in \X$. Moreover, for any functor $F:\mod R \to \mod R$, we denote by $F\X$ the subcategory of $\mod R$ consisting of all modules of the form $FX$, where $X\in \X$. \par

The essence of the next proposition is the Foxby equivalence given in \ref{prop2.2}. From this viewpoint, the assertion would be almost obvious since the resolvingness of subcategories is preserved by category equivalences.

\begin{prop} \label{prop3.1}
Let $C$ be a semidualizing $R$-module and $\X, \Y$ subcategories of $\mod R$. Then the following hold. 
\begin{enumerate}[\rm(1)]
\item
If $\X$ is contained in $\B_C$, then $\Hom_R(C,\res_C \X )= \res_R \Hom_R(C,\X)$. 
\item
If $\Y$ is contained in $\A_C$, then $C\otimes_R \res_R \Y = \res_C(C\otimes_R \Y)$.
\end{enumerate}
\end{prop}

\begin{proof}
We prove only part (1), as part (2) is proved similarly. \par  
First of all, we shall show that if $\X$ is contained in $\B_C$, then $ \C := \Hom_R(C,\res_C \X)$ is $R$-resolving. Since $R \cong \Hom_R(C,C) \in \C$, we have $R \in \C$. Suppose $Y$ is a direct summand of $X = \Hom_R(C,X_0)$ where $X_0$ is in $\res_C \X$. Since $\res_C \X$ is contained in $\B_C$, $X$ is in $\A_C$ and so is $Y$. Moreover, $C\otimes_R Y$ is a direct summand of $C\otimes_R X \cong X_0$. Hence $C\otimes_R Y \in \res_C \X$. 
Therefore we have  $Y \cong \Hom_R(C,C\otimes_R Y) \in \C$. 
This implies that $\C$ is closed under direct summands. Next, assume that $0\to X \to Y \to Z \to 0 $ is an exact sequence of $\mod R$ with $X=\Hom_R(C,X_0)$ , $X_0 \in \res_C \X $ and $ Z=\Hom_R(C,Z_0)$ , $Z_0 \in \res_C \X$. Note that the exact sequence above is in $\A_C$. Applying the functor $C\otimes_R -$, we have the exact sequence $0\to X_0 \to C\otimes_R Y \to Z_0 \to 0 $. Hence $C\otimes_R Y  \in \res_C \X$ and we have $Y \cong \Hom_R(C,C\otimes_R Y) \in \C $. This implies that $\C$ is closed under extensions. Similarly, we can prove that $\C$ is closed under kernels of epimorphisms and $\C$ is $R$-resolving. Dually, if $\Y$ is contained in $\A_C$, then $C\otimes_R \res_R \Y$ is $C$-resolving. \par
Finally, we shall show the minimality of $\Hom_R(C,\res_C \X) $. Suppose $\D$ is an $R$-resolving subcategory of $\mod R$ containing $\Hom_R(C,\X)$. Since $\Hom_R(C,\X)$ is contained in $\A_C$, we have $\Hom_R(C,\X) \subseteq \D \cap \A_C$. Hence we have $\X = C\otimes_R \Hom_R(C,\X) \subseteq C\otimes_R (\D \cap \A_C)$. As $\D \cap \A_C $ is an $R$- resolving subcategory contained in $\A_C$, $C\otimes_R (\D \cap \A_C)$ is $C$-resolving by the consequence of previous arguments. Therefore $\res_C \X \subseteq  C\otimes_R (\D \cap \A_C)$. Hence we have $\Hom_R(C,\res_C \X) \subseteq \Hom_R (C,C\otimes_R (\D \cap \A_C))= \D \cap \A_C \subseteq \D$. This shows that $\Hom_R(C,\res_C \X)$ is the smallest $R$-resolving subcategory of $\mod R$ containing $\Hom_R(C,\X)$. 
\end{proof}

The semidualizing version of the proof of \cite[Proposition 4.9]{Dao Takahashi} can be separated into several propositions as below. Some of them follow from the consequence of \cite {Dao Takahashi} and Foxby equivalence.

For a fixed semidualizing module $C$, we denote by $\nf_{C} (M)$ the subset of $\spec R$ consisting of prime ideals $\p$ such that  $M_{\p} \notin \add_{R_{\p}}C_{\p}$.

\begin{prop} \label{prop3.2}
Let $C$ be a semidualizing $R$-module and $M$ an $R$-module. Then there is an equality $\nf_C(M)=\nf_R(\Hom_R(C,M))$.
\end{prop}
\begin{proof} For any prime ideal $\p$ of $R$, one has $\pd_{C_{\p}} M_{\p} = \pd_{R_{\p}} \Hom_{R_{\p}}(C_{\p},M_{\p}) = \pd_{R_{\p}}\Hom_R(C,M)_{\p} $ where the first equality follows from \ref{prop2.3}. Hence we have $\p \notin \nf_C(M) $ if and only if $\pd_{C_{\p}} M_{\p} \leq 0$, if and only if $\pd_{R_{\p}}\Hom_R(C,M)_{\p} \leq 0$, if and only if $\p \notin \nf_R(\Hom_R(C,M)) $.
\end{proof}

\begin{prop} \label{prop3.3}
Let $C$ be a semidualizing $R$-module and $M$ an $R$-module with $0<\pd_C M <\infty $. Then there exists $N\in \res_C M$ satisfying $\nf_C(N)= \lbrace \m \rbrace $ and $\pd_C N =\pd_C M $. 
\end{prop}
\begin{proof}
Since $0<\pd_R \Hom_R(C,M) <\infty$, there exists $N_0 \in \res_R  \Hom_R (C,M) =  \Hom_R (C,\res_C M )$ such that $\pd_R N_0 = \pd_R \Hom_R(C,M)$, and $\nf_R(N_0)=\lbrace\m\rbrace$ by \cite[Lemma 4.6 and proof of Proposition 4.9]{Dao Takahashi}. Write $N_0 = \Hom_R(C,N)$ with $N\in \res_C M$, then we have $\nf_C(N)=\nf_R(\Hom_R(C,N))=\nf_R(N_0)=\lbrace\m\rbrace$, and $\pd_C N=\pd_R \Hom_R(C,N) =\pd_R N_0 =\pd_R \Hom_R(C,M) =\pd_C M$.
\end{proof}

\begin{prop} \label{prop3.4}
Let $C$ be a semidualizing $R$-module and $M$ an $R$-module. Then the following hold.\\
$(1)$ If M belongs to $\B_C$, then $\Ext^i_R(M,C) \cong \Ext^i_R(\Hom_R(C,M),R)$ for all $i \geq 0$. \\
$(2)$ If $n = \pd_C M $ is finite, then $\Ext^n_R(M,C)\neq 0 $.
\end{prop}
\begin{proof}
(1) For any integer $i \geq 0$, we have the following isomorphisms: $\Ext^i_R(\Hom_R(C,M),R) \cong \Ext^i_R(\Hom_R(C,M),\Hom_R(C,C)) \cong \Ext^i_R(C\otimes_R \Hom_R(C,M),C) \cong \Ext^i_R(M,C) $.\\
(2) Since  the projective  dimension of $\Hom_R (C,M)$ is finite, it equals to the supremum of non-negative integer $n$ such that $\Ext^n_R(\Hom_R(C,M),R) \neq 0$.
\end{proof}

\begin{prop} \label{prop3.5}
Let $C$ be a semidualizing $R$-module and $M$ an $R$-module. Suppose that $\Ext^1_R(M,C)$ is a nonzero $R$-module of finite length. Then there exists $N\in \res_C M $ satisfying $1\leq \pd_C N \leq \pd_C M$, and $\Ext^1_R(N,C) \cong k$.
\end{prop}

\begin{proof}
If $\l(\Ext^1_R(M,C))= 1$, then we can take $N:=M$. Suppose $\l(\Ext^1_R(M,C))> 1$. 
By hypothesis, we can choose a socle element $0\neq\sigma\in\Ext^1_R(M,C)$. It can be represented as a short exact sequence
\begin{center}
$\sigma:0 \to C \to N_1 \to M \to 0$. \\
\end{center}
Then $N_1\in\res_C M $. Since $\Ext^1_R(C,C)=0$ and $M\neq0$,  we have $\pd_C N_1 \leq\pd_C M $. Applying the functor $\Hom_R(-,C)$, we get an exact sequence 
\begin{center}
$\Hom_R(C,C) \xrightarrow{f} \Ext^1_R(M,C) \to \Ext^1_R(N_1,C) \to 0$ 
\end{center}
where $f$ sends $\id_C$ to $\sigma\in \Ext^1_R(M,C)$. Hence we obtain an exact sequence 
\begin{center}
$0\to k \to \Ext^1_R(M,C) \to \Ext^1_R(N_1,C) \to 0$. \\
\end{center}
This implies that $\l(\Ext^1_R(N_1,C))= \l(\Ext^1_R(M,C))- 1$. Replacing $M$ by $N_1$ and repeating this process, we get a family of $R$-modules $\lbrace N_i\rbrace^l_{i=1}$ such that $l=\l(\Ext^1_R(M,C))$ , $\pd_C N_{i+1} \leq \pd_C N_i $, $N_{i+1}\in \res_C N_i $, and $\l(\Ext^1_R(N_i,C))=l-i$ for all $i>0$. Hence we can take $N=N_{l-1}$.
\end{proof}

\begin{rmk}
If $M$ belongs to $\B_C$, the assertion follows from the consequence of \cite[Proof of Proposition 4.9]{Dao Takahashi} and Foxby equivalence as follows. Since $\Ext^1 _R (\Hom_R (C,M),R) \cong \Ext^1 _R (M,C)$ is nonzero $R$-module of finite length, there exists $N_0 \in \res_R \Hom_R(C,M) = \Hom_R(C,\res_C M)$ satisfying $\pd_R N_0 \leq \pd_R \Hom_R(C,M)$ and $\Ext^1_R(N_0,R) \cong k$ by \cite[Proof of Proposition 4.9]{Dao Takahashi}. Write $N_0 = \Hom_R(C,N)$ with $N\in \res_C M$, then we have $\pd_C N = \pd_R N_0 \leq \pd_C M $ and $\Ext^1_R(N,C) \cong \Ext^1 _R(N_0, R) \cong k$. 
\end{rmk}

\begin{prop} \label{prop3.6}
Let $C$ be a semidualizing $R$-module and $M$ an $R$-module satisfying $\pd_C(M)=1$, and $\Ext^1_R(M,C)\cong k$. Then for any $R$-module $L$ of finite length, $\tr_C L$ belongs to $\res_C M $.
\end{prop}
\begin{proof}
Since $M$ has finite $C$-projective dimension, it belongs to $\B_C$. Hence by \ref{prop3.4}, there are isomorphisms $\Ext^1_R(\Hom_R(C,M),R) \cong \Ext^1_R(M,C) \cong k$. Now $\Hom_R(C,M)$ has projective dimension one, and by \cite [Proof of Proposition 4.9]{Dao Takahashi}, we observe that $\tr_R L$ is in $\res_R  \Hom_R(C,M)= \Hom_R(C,\res_C M)$. Applying the functor $C\otimes_R -$, we see that $\tr_C L$ belongs to $\res_C M$.
\end{proof}

The next proposition plays an essential role in the proof of the divergence of the radius of a subcategory and its techniques of proof are based on \cite[Proof of Proposition 4.9]{Dao Takahashi}. 
\begin{prop} \label{prop3.7}
Let $C$ be a semidualizing $R$-module and $\X$ a subcategory of $\mod R$. Suppose that $\tr_C(R/\m^i) \in \X$ for all $i>0$, and $\depth C>0 $. Then $\radius \X = \infty$.
\end{prop}
\begin{proof}
Assume that we have $\radius\X = r < \infty $. We want to deduce a contradiction. There is a ball $\lbrack G \rbrack^R _{r+1}$ that contains it. Since $\X$ contains $\tr_C(R/\m^i)$ for all $i>0$, the ball $\lbrack G \rbrack^R _{r+1}$ also contains it. Taking the completions, we have $\tr_{\widehat{C}} (\widehat{R}/\m^{i}\widehat{R})\in \lbrack \widehat{G} \rbrack^{\widehat{R}} _{r+1} $ for all $i>0$. By virtue of Cohen's structure theorem, there exists a surjective homomorphism $S\to \widehat{R}$ such that $S$ is Gorenstein local ring with $\dim S = \dim \widehat R =: d$. Let $\n$ denote the maximal ideal of $S$ and note that we have $\widehat{R}/\m^{i}\widehat{R}=\widehat{R}/\n^{i}\widehat{R} $ for $i>0$. By \cite[Lemma 4.7]{Dao Takahashi}, there is an inclusion relation 
\begin{center}
$\displaystyle\bigcap_{j>0}\Ann_S\Ext^j _S (\tr_{\widehat{C}} (\widehat{R}/\n^{i}\widehat{R}),S)\supseteq \left( \displaystyle\prod_{j=1}^{d}\Ann_S\Ext^{j}_{S}(\widehat{R},S)\cdot \Ann_S\Ext^{j}_{S}(\widehat{G},S) \right) ^{r+1}$.
\end{center} \noindent
Fix an integer $i>0$ and let $a_1, \ldots, a_m$ be a system of generators of the ideal $\n^i$ of $S$. There is an exact sequence $\widehat{R}^{\oplus m} \xrightarrow{(a_1, \ldots, a_m)} \widehat{R} \to \widehat{R}/\n^{i}\widehat{R} \to 0 $ of $\widehat{R}$-modules. Dualizing this by $\widehat{C}$ induces an exact sequence
\begin{center}
$0=\Hom_{\widehat{R}}(\widehat{R}/\n^{i}\widehat{R},\widehat{C}) \to \widehat{C} \xrightarrow{\tiny\begin{pmatrix}a_1 \\ \vdots \\ a_m \end{pmatrix}} \widehat{C}^{\oplus m}\to \tr_{\widehat{C}} (\widehat{R}/\n^{i}\widehat{R}) \to 0 $
\end{center} \noindent
where the equality follows since $\depth \widehat{C} = \depth C >0.$ This makes an exact sequence
\begin{center}
$\Hom_S(\widehat{C},S)^{\oplus m} \xrightarrow{(a_1, \ldots, a_m)} \Hom_S(\widehat{C},S)\to \Ext^1_S(\tr_{\widehat{C}} (\widehat{R}/\n^{i}\widehat{R}),S) $, 
\end{center} \noindent
which yields an injection $\Hom_S(\widehat{C},S)/\n^i \Hom_S(\widehat{C},S) \to \Ext^1_S(\tr_{\widehat{C}} (\widehat{R}/\n^{i}\widehat{R}),S) $. \\
Thus \\
\begin{align*}
\Ann_S (\Hom_S(\widehat{C},S)/\n^{i} \Hom_S(\widehat{C},S)) 
 &\supseteq \Ann_S \Ext^1_S(\tr_{\widehat{C}} (\widehat{R}/\n^{i}\widehat{R}),S) \\ 
 &\supseteq \displaystyle\bigcap_{j>0}\Ann_S\Ext^j _S (\tr_{\widehat{C}} (\widehat{R}/\n^{i}\widehat{R}),S)  \\
&\supseteq \left( \displaystyle\prod_{j=1}^{d}\Ann_S\Ext^{j}_{S}(\widehat{R},S)\cdot \Ann_S\Ext^{j}_{S}(\widehat{G},S) \right) ^{r+1} 
\end{align*}
\vskip.5\baselineskip\noindent
and we obtain \vskip.5\baselineskip\noindent
\begin{align*}
\Ann_S \Hom_S(\widehat{C},S)&= \displaystyle\bigcap_{i>0} \Ann_S (\Hom_S(\widehat{C},S)/\n^{i} \Hom_S(\widehat{C},S))  \\
&\supseteq \left( \displaystyle\prod_{j=1}^{d}\Ann_S\Ext^{j}_{S}(\widehat{R},S)\cdot \Ann_S\Ext^{j}_{S}(\widehat{G},S) \right) ^{r+1} 
\end{align*}
 \vskip.5\baselineskip\noindent
where the equality follows from \cite[Lemma 4.8]{Dao Takahashi}. This implies \\
\begin{center}
$\supp_S \Hom_S(\widehat{C},S) \subseteq \displaystyle \bigcup_{j=1}^{d} (\supp_S \Ext^{j}_{S}(\widehat{R},S) \cup \supp_S \Ext^{j}_{S}(\widehat{G},S))$. \par
\end{center}

Let $I$ be the kernel of the surjection $S\to \widehat{R}$. Since $\dim S = d = \dim \widehat{R}$, the ideal $I$ of $S$ has height zero. Hence there exists a minimal prime ideal $\p$ of $S$ which contains $I$. Note that $\widehat{C}$ is a semidualizing $\widehat{R}$-module. Since $S_{\p}$ is an artinian Gorenstein local ring and $\widehat{C}_{\p} \cong \widehat{C}_{\p \widehat{R}} \neq 0$, the $S_{\p}$-module $\Hom_S(\widehat{C},S)_{\p} \cong \Hom_{S_{\p}}(\widehat{C}_{\p}, S_{\p})$ is nonzero. This implies that $\p$ belongs to $ \supp_S \Hom_S(\widehat{C},S) $. Therefore, for some integer $1\leq l \leq d$, the prime ideal $\p$ is in $\supp_S \Ext^{l}_{S}(\widehat{R},S) $ or $ \supp_S \Ext^{l}_{S}(\widehat{G},S) $. This contradicts the fact that $S_{\p}$ is injective $S_{\p}$-module. This contradiction proves that $\radius \X = \infty$. 
\end{proof} 

Combining the propositions above, we get our most general result.
\begin{thm} \label{thm3.1}
Let $C$ be a semidualizing $R$-module and $\X$ a $C$-resolving subcategory of $\mod R$. If $\X$ contains an $R$-module $M$ such that $0<\pd_C M <\infty$, then $\radius \X = \infty$.
\end{thm}
\begin{proof} Write $n= \pd_C M$, then there exists an exact sequence $0\to C_n \xrightarrow{\delta_n} C_{n-1} \xrightarrow{\delta_{n-1}} \cdots \xrightarrow{\delta_2} C_1 \xrightarrow{\delta_1} C_0 \to M \to 0$ such that each $C_i$ belongs to $\add_R C$. Since $\image \delta_{n-1} \in \X$ and $\pd_C (\image \delta_{n-1}) = 1$, we may assume that $\pd_C M = 1$. Then by \ref{prop3.3} and \ref{prop3.5}, we may further assume that $\Ext_R (M,C) \cong k$. Hence by \ref{prop3.6} and \ref{prop3.7}, we have $\radius \X = \infty$. 
\end{proof}

Letting $C=R$ in Theorem \ref{thm3.1}, we immediately recover \cite[Proposition 4.9]{Dao Takahashi}.

\begin{cor}[Dao and Takahashi]
Let $\X$ be a resolving subcategory of $\mod R$. If $\X$ contains an $R$-module $M$ such that $0<\pd_R M < \infty$, then $\radius \X = \infty$.
\end{cor}

Applying the above theorem to the case where $R$ is Cohen--Macaulay and $C$ is a canonical module, we immediately get the following corollary, which is the main result of this paper.

\begin{cor} \label{cor3.2}
Let $R$ be a Cohen--Macaulay local ring with a canonical module $\omega$. Let $\X$ be an  $\omega$-resolving subcategory of $\mod R$ (e.g., a resolving subcategory of $\mod R$ containing $\omega$). If $\X$ contains an $R$-module $M$ such that $\id_R M< \infty$, and $M$ is non-maximal Cohen--Macaulay module. Then $\radius \X = \infty$.
\end{cor}

\begin{proof}
By \cite[3.3.28]{BH}, we have $0<\pd_\omega M<\infty$. Hence we obtain $\radius\X = \infty$ by Theorem \ref{thm3.1}. 
\end{proof}

\begin{rmk} \label{rmk3.1}
In the situation of \ref{thm3.1}, the divergence of $\radius_C \X$  holds immediately as follows. Since $\Hom_R(C,X)$ is in $\Hom_R(C,\X \cap \B_C)$, the subcategory $\Hom_R(C,\X \cap \B_C)$ is $R$-resolving. Also, we have $0<\pd_R \Hom_R(C,X) <\infty$ and $\radius \Hom_R(C,\X \cap \B_C)= \infty$ by \cite[Proposition 4.9]{Dao Takahashi}. Now we get the inequalities $\radius_C \X \geq \radius_C (\X \cap \B_C) \geq \radius\Hom_R(C, \X \cap \B_C)$, where the second inequality follows from the fact that for a subcategory $\X$ of $\mod R$, $G \in \B_C$ and $r\ge 1$ the equivalence 
\begin{center}
$\X \subseteq [G]_r^C \iff \Hom_R(C,\X) \subseteq [\Hom_R(C,G)]_r^R$ 
\end{center}
holds.
Hence we have $\radius_C \X = \infty$. However, this consequence does not imply the conclusion of \ref{thm3.1}, because the module $G$ in \eqref{radiusdef} does not necessarily belong to $\B_C$ nor $\A_C$. So we need to establish the semidualizing version of the proof of \cite[Proposition 4.9]{Dao Takahashi}. 
\end{rmk}

\begin{rmk}
When we generalize some statements by replacing $R$ with a semidualizing module $C$, we may expect that the new statements provide unknown meaningful consequences even if $C$ is the canonical module of a Cohen--Macaulay local ring $R$. From this viewpoint, the direct generalization of \cite[Theorem 4.10]{Dao Takahashi} is unlikely to be done unless we find another way. Indeed, if $R$ is Cohen--Macaulay with canonical module $\omega$, every $R$-module $M$ has finite $\tG _{\omega}$-dimension. So the hypothesis of \cite[Theorem 4.10]{Dao Takahashi} has no meaning. Moreover, when we try to use its proof by replacing $R$ by a semidualizing module $C$, it is necessary to assume $\pd_C R <\infty$. Hence there exists a surjection $C^{\oplus m} \to R$ and this map is a split epimorphism. Therefore $R \in \add_R C$. Since $R$ is indecomposable, we have $R \cong C$. This means that we cannot assume $\pd_C R <\infty$.
\end{rmk}

\section{Applications of main theorem}

The purpose of this section is to state and prove Theorem \ref{thm4.1} below, as an application of Corollary \ref{cor3.2}. To achieve this purpose, we state a couple of lemmas.

\begin{lem} \label{lem4.1}
Let $R$ be a $d$-dimensional Cohen--Macaulay local ring, $X$ a nonzero maximal Cohen--Macaulay $R$-module with $\id_R X < \infty$, and $M$ an $R$-module with $\id_R X < \infty$. Then the following hold.
\begin{enumerate}[\rm(1)]
\item
One has $\Tor^R _i (X,\Hom_R (X,M))=0$ for all $i>0$ and $\pd_R \Hom_R (X,M)\leq d-\depth M$.
\item 
Suppose that $M$ is nonzero and set $n=d-\depth M$. Then there exists an exact sequence $0 \to X^{\oplus a_n} \to \cdots \to X^{\oplus a_0} \to M^{\oplus m} \to 0$, where $a_0, \dotsc, a_n$ are non-negative integers, and $m$ is a positive integer.
\end{enumerate}
\end{lem}
\begin{proof}
(1) We may assume that $R$ is complete. Let $\omega$ be a canonical module of $R$. Since $X$ is in $\add \omega$ and $M$ is in $\B_{\omega}$, we have $\Tor^R _i (X,\Hom_R (X,M))=0$ for all $i>0$. On the other hand, by \cite[3.3.28]{BH}, there exists an exact sequence $0 \to \omega_n \to \cdots \to \omega_0 \to M \to 0$ in $\B_{\omega}$, where $\omega_i \in \add \omega$ for all $i$ and $n=d-\depth M$. Applying the functor $\Hom_R(X,-)$, we have an exact sequence $0 \to \Hom_R(X,\omega_n) \to \cdots \to \Hom_R(X,\omega_0) \to \Hom_R(X,M) \to 0$. Since $\Hom_R(X,\omega_i)$ is a free $R$-module for each $0 \leq i \leq n$, the sequence $0 \to \Hom_R(X,\omega_n) \to \cdots \to \Hom_R(X,\omega_0) \to 0$ gives a free resolution of $\Hom_R(X,M)$.  \\
(2) Let $\omega_{\widehat{R}}$ be a canonical module of $\widehat{R}$. The $R$-module $\Hom_R(X,X)$ is free; let $m$ be its free rank. Note then that $m>0$. There exist isomorphisms $X\otimes_R \Hom_R(X,M) \cong \Hom_R(\Hom_R(X,X),M) \cong M^{\oplus m}$, since $\widehat{X}$ is in $\add \omega_{\widehat{R}}$ and $\widehat{M}$ is in $\B_{\omega_{\widehat{R}}}$. By (1), we can take a free resolution $0 \to R^{\oplus a_n} \to \cdots \to R^{\oplus a_0} \to 0$ of $\Hom_R(X,M)$.  Applying the functor $X \otimes_R -$, we have an exact sequence $0 \to X^{\oplus a_n} \to \cdots \to X^{\oplus a_0} \to M^{\oplus m} \to 0$.
\end{proof}

Note that for any subcategory $\X$ of $\mod R$ and a flat $R$-algebra $S$, we have the inequality $\radius_S (\X \otimes_R S) \leq \radius \X$. A special case can be found in \cite[Remark 4.1]{Dao Takahashi}.

\begin{lem} \label{lem4.2}
Let $R$ be a Cohen--Macaulay local ring and $\X$ a subcategory of $\mod R$ satisfying the following conditions: 
\begin{enumerate}[\rm(1)]
\item
The subcategory $\X$ of $\mod R$ is closed under direct summands, extensions, and kernels of epimorphisms.
\item 
The modules in $\X$ are locally maximal Cohen--Macaulay on the punctured spectrum of $R$ and have finite injective dimension.
\item
The subcategory $\X$ contains two $R$-modules $X,Y$ such that $X$ is a nonzero maximal Cohen--Macaulay module and $Y$ is not a maximal Cohen--Macaulay module.
\end{enumerate}
Then $\add_{\widehat{R}} \widehat{\X}$ is $\omega_{\widehat{R}}$-resolving and $\radius \X = \infty$.
\end{lem}
\begin{proof}
Since $\add_{\widehat{R}} \widehat{\X}$ is closed under direct summands and $\omega_{\widehat{R}}^{\oplus m} \cong \widehat{X} \in \widehat{\X}$ with $m>0$, we have $\omega_{\widehat{R}} \in \add_{\widehat{R}} \widehat{\X}$. We shall show that $\add_{\widehat{R}} \widehat{\X}$ is closed under extensions. We have only to prove that if there exists an exact sequence $ \sigma : 0 \to \widehat{A} \to M \to \widehat{B} \to 0$ such that $A,B \in \X$, then $M$ is in $\widehat{\X}$. By the hypothesis (2), the module $\Ext^1 _R (B,A)$ has finite length. Hence $\Ext^1 _R (B,A)$ is complete. Thus there exists an exact sequence $\sigma_0 : 0 \to A \to N \to B \to 0 $ such that $\widehat{\sigma_0} = \sigma$. Since $\X$ is closed under extensions, we have $N \in \X$ and therefore $M \cong \widehat{N} \in \widehat{\X}$. \par
Next, we will show that $\add_{\widehat{R}} \widehat{\X}$ is closed under kernels of epimorphisms. We have only to prove that if there exists an exact sequence $0 \to M \to \widehat{A} \to \widehat{B} \to 0$ such that $A,B \in \X$, then $M$ is in $\add_{\widehat{R}} \widehat{\X}$. By \ref{lem4.1}, we have an exact sequence $0 \to C \to X^{\oplus n} \to B^{\oplus m} \to 0$. Since $\X$ is closed under kernels of epimorphisms, $C$ is in $\X$. By taking completion and direct sums, we have the following pullback diagram: 
\begin{equation*}
\xymatrix{
  &   & 0 \ar[d] & 0 \ar[d] &   \\
  &   & \widehat{C} \ar@{=}[r] \ar[d] & \widehat{C} \ar[d] &   \\
0 \ar[r] & M^{\oplus m} \ar[r] \ar@{=}[d] & D \ar[r] \ar[d] & \widehat{X}^{\oplus n} \ar[r] \ar[d] & 0 \\
0 \ar[r] & M^{\oplus m} \ar[r] & \widehat{A}^{\oplus m} \ar[r] \ar[d] & \widehat{B}^{\oplus m} \ar[r] \ar[d] & 0 \\
  &   & 0 & 0 &   
}
\end{equation*}
We see from the middle column that $D$ is in $\widehat{\X}$. Since $M$ has finite injective dimension, $M$ belongs to $\B_{\omega_{\widehat R}}$. Hence we have $\Ext^1_{\widehat{R}}(\widehat{X},M) = 0$. This implies that the first row is a split exact sequence. Thus $M$ is a direct summand of $D$ and therefore $M$ is in $\add_{\widehat{R}} \widehat{\X}$. Now we proved that $\add_{\widehat{R}} \widehat{\X}$ is $\omega_{\widehat{R}}$-resolving. Hence by \ref{cor3.2}, we have $\infty = \radius _{\widehat{R}}( \add_{\widehat{R}} \widehat{\X} ) \leq \radius \X$.
\end{proof}

Now we can achieve the purpose of this section.

\begin{thm} \label{thm4.1}
Let $R$ be a Cohen--Macaulay local ring, and $\X$ an $R$-resolving subcategory of $\mod R$. Suppose that $\X$ contains two $R$-modules $X,Y$ of finite injective dimension such that $X$ is a nonzero maximal Cohen--Macaulay module and $Y$ is not a maximal Cohen--Macaulay module. Then $\radius \X = \infty$. 
\end{thm}
\begin{proof}
Let $\Phi$ be the set of prime ideals $\p$ of $R$ such  that $Y_{\p}$ is not a maximal Cohen--Macaulay $R_{\p}$ -module. Take a minimal element $\p$ of $\Phi$ with respect to the inclusion relation.
By \cite[Lemma 4.8]{stcm}, the subcategory $\add_{R_{\p}}\X_{\p}$ is an $R_{\p}$-resolving subcategory of $\mod R_{\p}$. We have $\supp X = \spec R$ since $\Hom_R(X,X)$ is a nonzero free $R$-module. Therefore we may assume that $Y$ is locally maximal Cohen--Macaulay on the punctured spectrum. Then the modules in $\seek\{X,Y\}$ are locally maximal Cohen--Macaulay on the punctured spectrum and have finite injective dimension. Hence by \ref{lem4.2}, we have $\infty = \radius(\seek\{X,Y\}) \leq \radius \X$.
\end{proof}

\begin{rmk}
Theorem \ref{thm4.1} assumes that $\X$ is $R$-resolving, so that it does not recover Corollary \ref{cor3.2}. Lemma \ref{lem4.2} and the proof of Theorem \ref{thm4.1} are due to Kaito Kimura. The author is indebted to him for his offer to include them in this paper.
\end{rmk}

\begin{ac}
The author would like to thank his supervisor Ryo Takahashi for his careful reading of this manuscript and many thoughtful suggestions. He also thanks Yuya Otake and Kaito Kimura for their valuable comments.
\end{ac}

\end{document}